\documentclass[preprint,12pt]{elsarticle}%
\usepackage{amssymb}
\usepackage{amsmath}
\usepackage{amsfonts}
\usepackage{graphicx}%
\providecommand{\U}[1]{\protect \rule{.1in}{.1in}}
\graphicspath{{C:/Figures/}}
\providecommand{\U}[1]{\protect \rule{.1in}{.1in}}

\newtheorem{corollary}{Corollary}

\newtheorem{proposition}{Proposition}

\newtheorem{theorem}{Theorem}
\numberwithin{equation}{section}
\newenvironment{proof}{\paragraph{Proof:}}{\hfill$\square$}
\journal{...}
\begin{document}
\begin{frontmatter}
\title{Global existence of solutions for an $m$-component reaction--diffusion system with a tridiagonal 2-Toeplitz diffusion matrix
and polynomially growing reaction terms}
\author[Tai1]{Salem Abdelmalek}
\author[Tai2]{Samir Bendoukha}
\address[Tai1]{Department of Mathematics, College of Sciences, Yanbu Taibah
University, Saudi Arabia. Email: sabdelmalek@taibahu.edu.sa \newline Department of Mathematics, University of Tebessa 12002 Algeria.}
\address[Tai2]{Department of Electrical Engineering, College of Engineering,
Yanbu, Taibah University, Saudi Arabia. Email: sbendoukha@taibahu.edu.sa}
\begin{abstract}
This paper is concerned with the local and global existence of solutions for a generalized $m$-component reaction--diffusion
system with a tridiagonal $2$--Toeplitz diffusion matrix and polynomial growth. We derive the eigenvalues and
eigenvectors and determine the parabolicity conditions in order to diagonalize the proposed system. We, then,
determine the invariant regions and utilize a Lyapunov functional to establish the global existence of solutions for the
proposed system. A numerical example is used to illustrate and confirm the findings of the study.
\end{abstract}
\end{frontmatter}

\section{\textbf{Introduction}\label{SecModel}}

In this study, we consider the generalized $m$-component reaction--diffusion
system with $m\geq2$:%
\begin{equation}
\dfrac{\partial U}{\partial t}-A\Delta U=F\left(  U\right)  , \label{Sys1}%
\end{equation}
in $\Omega \times \left(  0,+\infty \right)  $, where $\Omega$ is an open bounded
domain of class $C^{1}$ in $\mathbb{%
\mathbb{R}
}^{m}$ with boundary $\partial \Omega$. The diffusion matrix $A$ is assumed to
be of the form%
\begin{equation}
A=\left(
\begin{array}
[c]{cccccc}%
\alpha_{1} & \gamma_{1} & 0 & \cdots & \cdots & 0\\
\beta_{1} & \alpha_{2} & \gamma_{2} & \ddots &  & \vdots \\
0 & \beta_{2} & \alpha_{1} & \gamma_{1} & \ddots & \vdots \\
\vdots & \ddots & \beta_{1} & \alpha_{2} & \gamma_{2} & 0\\
\vdots &  & \ddots & \beta_{2} & \ddots & \ddots \\
0 & \cdots & \cdots & 0 & \ddots & \ddots
\end{array}
\right)  _{m\times m}, \label{TriMat1}%
\end{equation}
with $\alpha_{1},\alpha_{2},\beta_{1},\beta_{2},\gamma_{1},\gamma_{2}>0$ being
positive real numbers representing the self and cross--diffusion constants and
satisfying the inequality%
\begin{equation}
\frac{\sqrt{\alpha_{1}\alpha_{2}}}{\max \left \{  \beta_{1}+\gamma_{1},\beta
_{2}+\gamma_{2}\right \}  }>\cos \left(  \frac{\pi}{m+1}\right)  . \label{Para1}%
\end{equation}

The Laplacian operator $\Delta=\overset{M}{\underset{i=1}{\sum}}\frac
{\partial^{2}}{\partial x_{i}^{2}}$ has a spatial dimension of $M$ and
$F\left(  U\right)  $ is a polynomially growing functional representing the
reaction terms of the system.

The boundary conditions and initial data for the proposed system are assumed
to satisfy%
\begin{equation}
\alpha U+\left(  1-\alpha \right)  \partial_{\eta}U=B\text{\  \  \  \ on }%
\partial \Omega \times \left(  0,+\infty \right)  , \label{Bound1}%
\end{equation}
or%
\begin{equation}
\alpha U+\left(  1-\alpha \right)  A\partial_{\eta}U=B\text{\  \  \  \ on
}\partial \Omega \times \left(  0,+\infty \right)  \label{Bound2}%
\end{equation}
and%
\begin{equation}
U\left(  x,0\right)  =U_{0}\left(  x\right)  \text{ \  \  \ on}\; \Omega,
\label{Init1}%
\end{equation}
respectively. For generality, we will consider three types of boundary
conditions in this paper:

\begin{enumerate}
\item[(i)] Nonhomogeneous Robin boundary conditions, corresponding to%
\[
0<\alpha<1,\text{ }B\in%
\mathbb{R}
^{m};
\]

\item[(ii)] Homogeneous Neumann boundary conditions, corresponding to%
\[
\alpha=0\text{ and }B\equiv0;
\]

\item[(iii)] Homogeneous Dirichlet boundary conditions, corresponding to%
\[
1-\alpha=0\text{ and }B\equiv0.
\]

\end{enumerate}

Note that $\dfrac{\partial}{\partial \eta}$ denotes the outward normal
derivative on $\partial \Omega$ and the vectors $U$, $F$, and $B$ are defined
as\
\begin{align*}
U  &  :=\left(  u_{1},...,u_{m}\right)  ^{T},\\
F  &  :=\left(  f_{1},...,f_{m}\right)  ^{T},\\
B  &  :=\left(  \beta_{1},...,\beta_{m}\right)  ^{T}.
\end{align*}
The initial data is assumed to be in the region given by%
\begin{equation}
\Sigma_{\mathfrak{L},\emptyset}=\left \{  U_{0}\in \mathbb{%
\mathbb{R}
}^{m}:\left \langle V_{\ell},U_{0}\right \rangle \geq0,\text{ }\ell
\in \mathfrak{L}\right \}  , \label{InvReg1}%
\end{equation}
subject to%
\begin{equation}
\left \langle V_{\ell},B\right \rangle \geq0,\ell \in \mathfrak{L.}
\label{InvReg2}%
\end{equation}

The study at hand builds upon numerous previous works found in the literature.
Among the most relevant studies is that of Abdelmalek in \cite{Abdelmalek2016}
where he considered an $m$-component tridiagonal matrix of the form%
\[
A=\left(
\begin{array}
[c]{ccccc}%
\alpha & \gamma & 0 & \cdots & 0\\
\beta & \alpha & \gamma & \ddots & \vdots \\
0 & \beta & \ddots & \ddots & 0\\
\vdots & \ddots & \ddots & \ddots & \gamma \\
0 & \cdots & 0 & \beta & \alpha
\end{array}
\right)  _{m\times m},
\]
and proved the global existence of solutions subject to the parabolicity
condition%
\[
\frac{\alpha}{\beta+\gamma}>\cos \frac{\pi}{m+1},
\]
which can be easily shown to fall under the general condition in (\ref{Para1})
with $\alpha=\alpha_{1}=\alpha_{2}$, $\beta=\beta_{1}=\beta_{2}$, and
$\gamma=\gamma_{1}=\gamma_{2}$.

Another important study is that of Kouachi and Rebiai in \cite{Kouachi2010}
where the authors established the global existence of solutions for a
$3\times3$ tridiagonal $2$--Toeplitz matrix of the form%
\[
A=\left(
\begin{array}
[c]{ccc}%
\alpha_{1} & \gamma_{1} & 0\\
\beta_{1} & \alpha_{2} & \gamma_{2}\\
0 & \beta_{2} & \alpha_{1}%
\end{array}
\right)  ,
\]
subject to the parabolicity condition%
\[
2\sqrt{\alpha_{1}\alpha_{2}}>\sqrt{\left(  \beta_{1}+\gamma_{1}\right)
^{2}+\left(  \beta_{2}+\gamma_{2}\right)  ^{2}}.
\]

Note that this condition is weaker than%
\[
\sqrt{2\alpha_{1}\alpha_{2}}>\max \left \{  \beta_{1}+\gamma_{1},\beta
_{2}+\gamma_{2}\right \}  ,
\]
which is obtained from (\ref{Para1}) for $m=3$. Although the work carried out
in \cite{Kouachi2010} is important to us here, it is necessary to note that
the authors failed to identify all the invariant regions of the proposed
system and settled for only $4$ of them.

This paper will build upon the work of these two studies by assuming the
diffusion matrix to be $m$--component tridiagonal $2$--Toeplitz and
determining all the possible invariant regions for the system. A Lyapunov
functional will be used to establish the global existence of solutions in
these regions.

The remainder of this paper is organized as follows: Section \ref{SecEigen}
uses the three point Chebyshev recurrence relationhip of polynomials to derive
the eigenvalues and eigenvectors of the transposed dffusion matrix for the odd
and even dimension cases, respectively. Section \ref{SecPara} derives the
parabolicity conditions for the proposed system, which is essential for the
diagonalization process, which follows in Section \ref{SecMain}. Section
\ref{SecMain} shows how the invariant regions of the equivalent digonalized
system can be identified and proves the local and global existence of
solutions. The last section of this paper will present a confirmation and
validation of the findings through the use of numerical examples solved by
means of the finite difference approximation method.

\section{Eigenvalues and Eigenvectors\label{SecEigen}}

For reasons that will become apparent in the following section, we will first
derive the eigenvalues and eigenvectors of matrix $A^{T}$ with $A$ being the
proposed tridiagonal $2$--Toeplitz diffusion matrix. We refer to the work of
Gover in \cite{Gover1994} where the characteristic polynomial of a tridiagonal
$2$--Toeplitz matrix was shown to be closely connected to polynomials that
satisfy the three point Chebyshev recurrence relationship. First, we have%
\begin{equation}
A^{T}=\left(
\begin{array}
[c]{cccccc}%
\alpha_{1} & \beta_{1} & 0 & \cdots & \cdots & 0\\
\gamma_{1} & \alpha_{2} & \beta_{2} & \ddots &  & \vdots \\
0 & \gamma_{2} & \alpha_{1} & \beta_{1} & \ddots & \vdots \\
\vdots & \ddots & \gamma_{1} & \alpha_{2} & \beta_{2} & 0\\
\vdots &  & \ddots & \gamma_{2} & \ddots & \ddots \\
0 & \cdots & \cdots & 0 & \ddots & \ddots
\end{array}
\right)  _{m\times m}. \label{2.0}%
\end{equation}

The exact shape and characteristics of $A^{T}$ differ for odd and even values
of the dimension $m$. Hence, we will consider the two cases separately. Before
we present the main findings of \cite{Gover1994}, let us define the constants%
\begin{equation}
\beta=\sqrt{\frac{\beta_{2}\gamma_{2}}{\beta_{1}\gamma_{1}}}\text{ and
}s=\sqrt{\frac{\gamma_{1}\gamma_{2}}{\beta_{1}\beta_{2}}}. \label{2.3}%
\end{equation}

We also define the polynomials%
\begin{equation}
\left \{
\begin{array}
[c]{l}%
q_{0}\left(  \mu \right)  =1,\text{ }q_{1}\left(  \mu \right)  =\mu+\beta \\
q_{n+1}\left(  \mu \right)  =\mu q_{n}\left(  \mu \right)  -q_{n-1}\left(
\mu \right)  ,
\end{array}
\right.  \label{2.1}%
\end{equation}
and%
\begin{equation}
\left \{
\begin{array}
[c]{l}%
p_{0}\left(  \mu \right)  =1,\text{ }p_{1}\left(  \mu \right)  =\mu \\
p_{n+1}\left(  \mu \right)  =\mu p_{n}\left(  \mu \right)  -p_{n-1}\left(
\mu \right)  ,
\end{array}
\right.  \label{2.2}%
\end{equation}
whose zeros are denoted by $Q_{r}$ and $P_{r}$, respectively, for $r=1,...,n$.
We note that $p_{n}\left(  \mu \right)  $ is a Chebyshev polynomial of the
first kind, whereas $q_{n}\left(  \mu \right)  $ is not. As shown in
\cite{Gover1994}, the zeros of $p_{n}\left(  \mu \right)  $ can be given by
\[
P_{r}=2\cos \frac{r\pi}{n+1},
\]
whereas for $Q_{r}$ no explicit form was found.

Let us now summarize the eigenvalues and eigenvectors for the odd and even
cases separately. First, for $m=2n+1$, we obtain the following results:

\begin{theorem}
\label{OddEigVal}The eigenvalues of the matrix $A^{T}$ of order $m=2n+1$ given
in (\ref{2.0}) are $\alpha_{1}$ along with\ the solutions of the quadratic
equations%
\begin{equation}
\frac{\left(  \alpha_{1}-\lambda \right)  \left(  \alpha_{2}-\lambda \right)
}{\sqrt{\beta_{1}\beta_{2}\gamma_{1}\gamma_{2}}}-\frac{1}{\beta}-\beta=P_{r},
\label{2.4}%
\end{equation}
for $r=1,2,...,n$.
\end{theorem}

Note that for every $P_{r}$ there exist two eigenvalues for matrix $A^{T}$,
which along with $\alpha_{1}$\ yields $m=2n+1$ eigenvalues. For notational
purposes, let us define a duplicated set of zeros given by%
\[
P_{2r}^{^{\prime}}=P_{2r-1}^{^{\prime}}=P_{r},
\]
for $r=1,2,...,n$.

\begin{theorem}
\label{OddEigVec}The eigenvector of the matrix $A^{T}$ of order $m=2n+1$ given
in (\ref{2.0}) associated with the eigenvalue $\lambda_{r}$, for
$r=1,...,2n$,\ is given by%
\begin{equation}
V_{\lambda_{r}}=\left(  v_{1\lambda_{r}},v_{2\lambda_{r}},...,v_{m\lambda_{r}%
}\right)  ^{T}, \label{2.5}%
\end{equation}
where%
\begin{equation}
v_{\ell \lambda_{r}}=\left \{
\begin{array}
[c]{ll}%
s^{\frac{\ell-1}{2}}q_{^{\frac{\ell-1}{2}}}\left(  P_{r}^{^{\prime}}\right)
, & \ell \text{ is odd}\\
-\frac{1}{\beta_{1}}s^{\frac{\ell}{2}-1}\left(  \alpha_{1}-\lambda_{r}\right)
p_{\frac{\ell}{2}-1}\left(  P_{r}^{^{\prime}}\right)  , & \ell \text{ is even,}%
\end{array}
\right.  \label{2.5.1}%
\end{equation}
for\ $\ell=1,...,m$. The eigenvector associated with the eigenvalue
$\alpha_{1}$ is%
\begin{equation}
V_{\alpha_{1}}=\left(  v_{1\alpha_{1}},v_{2\alpha_{1}},...,v_{m\alpha_{1}%
}\right)  ^{T}, \label{2.5.2}%
\end{equation}
with%
\begin{equation}
v_{\ell \alpha_{1}}=\left \{
\begin{array}
[c]{ll}%
\left(  -\frac{\gamma_{1}}{\beta_{2}}\right)  ^{\frac{\ell-1}{2}}, &
\ell \text{ is odd}\\
0, & \ell \text{ is even,}%
\end{array}
\right.  \label{2.5.3}%
\end{equation}
for $\ell=1,...,m$.
\end{theorem}

The second case is where the matrix $A^{T}$ (\ref{2.0}) has an even dimension
$m=2n$. The following holds:

\begin{theorem}
\label{EvenEigVal}The eigenvalues of the matrix $A^{T}$ of order $m=2n+1$
given in (\ref{2.0}) denoted by $\lambda_{r}$\ are the solutions of the
quadratic equations%
\begin{equation}
\frac{\left(  \alpha_{1}-\lambda \right)  \left(  \alpha_{2}-\lambda \right)
}{\sqrt{\beta_{1}\beta_{2}\gamma_{1}\gamma_{2}}}-\frac{1}{\beta}-\beta=Q_{r},
\label{2.6}%
\end{equation}
for $r=1,2,...,n$, where $Q_{r}$ are the zeros of $q_{n}\left(  \mu \right)  $.
\end{theorem}

Similar to $P_{r}$, there exist two eigenvalues for matrix $A^{T}$ associated
with every value of $Q_{r}$, which yields $m$ eigenvalues. In order to
simplify the notation, we define the duplicated set of zeros given by%
\[
Q_{2r}^{^{\prime}}=Q_{2r-1}^{^{\prime}}=Q_{r},
\]
for $r=1,2,...,n$.

\begin{theorem}
\label{EvenEigVec}The eigenvector of the matrix $A^{T}$ of order $m=2n+1$
given in (\ref{2.0}) associated with the eigenvalue $\lambda_{r}$ is given by%
\begin{equation}
V_{r}=\left(  v_{1\lambda_{r}},v_{2\lambda_{r}},...,v_{m\lambda_{r}}\right)
^{T}, \label{2.7}%
\end{equation}
with%
\begin{equation}
v_{\ell \lambda_{r}}=\left \{
\begin{array}
[c]{ll}%
s^{\frac{\ell-1}{2}}q_{^{\frac{\ell-1}{2}}}\left(  Q_{r}^{^{\prime}}\right)
, & \ell \text{ is odd}\\
-\frac{1}{\beta_{1}}s^{\frac{\ell}{2}-1}\left(  \alpha_{1}-\lambda_{r}\right)
p_{\frac{\ell}{2}-1}\left(  Q_{r}^{^{\prime}}\right)  , & \ell \text{ is even,}%
\end{array}
\right.  \label{2.7.1}%
\end{equation}
for $\ell=1,...,m$.
\end{theorem}

\section{Parabolicity\label{SecPara}}

In this section, we will derive the parabolicity condition for the proposed
system. Parabolicity is crucial to the diagonalization process, which we will
be discussed later on in Section \ref{SecMain}. In order to ensure the
parabolicity of the system, we examine the positive definiteness of the
proposed diffusion matrix. Generally speaking, a matrix is said to be positive
definite if and only if its top-left corner principal minors are all positive.
To this end, Andelic and da Fonesca \cite{Andelic2011} and others examined the
parabolicity condition for a tridiagonal symmetric matrix. The following
theorem holds.

\begin{proposition}
\label{AndelicProp}Let $T$ be the tridiagonal matrix defined as%
\[
T=\left(
\begin{array}
[c]{ccccc}%
a_{1} & b_{1} & 0 & \cdots & 0\\
b_{1} & a_{2} & b_{2} &  & \vdots \\
0 & b_{2} & \ddots & \ddots & \vdots \\
\vdots &  & \ddots & \ddots & b_{m-1}\\
0 & \cdots & \cdots & b_{m-1} & a_{m}%
\end{array}
\right)
\]
with positive diagonal entries. If
\begin{equation}
a_{i}a_{i+1}>4b_{i}^{2}\cos^{2}\left(  \frac{\pi}{m+1}\right)
\label{AndelicCond}%
\end{equation}
for $i=1,...,m-1$, then $T$ is positive definite.
\end{proposition}

Since the diffusion matrix considered here is not symmetric, Proposition
\ref{AndelicProp} does not apply directly to it. However, we know that if a
matrix is not symmetric, its quadratic form $Q=\left \langle X,AX\right \rangle
=X^{T}AX$, with $X$ being an arbitrary column vector, is said to be positive
definite if and only if the principal minors in the top--left corner of
$\frac{1}{2}\left(  A+A^{T}\right)  $ are all positive. In order to derive
sufficient conditions for matrix $A$ in (\ref{TriMat1}), we apply Proposition
\ref{AndelicProp} to produce the following Theorem.

\begin{theorem}
Let $A$ be the tridiagonal $2$--Toeplitz matrix defined in (\ref{TriMat1}).
The quadratic form of $A$ is positive definite iff condition (\ref{Para1}) is
satisfied. It follows that subject to (\ref{Para1}), the reaction diffusion
system (\ref{Sys1}) satisfies the parabolicity condition.
\end{theorem}

\begin{proof}
Condition (\ref{AndelicCond}) can be rearranged to the form%
\begin{equation}
\sqrt{a_{i}a_{i+1}}>2\left \vert b_{i}\right \vert \cos \left(  \frac{\pi}%
{m+1}\right)  . \label{Para2}%
\end{equation}
The symmteric counterpart of $A$\ as defined in (\ref{TriMat1}) can be given
by%
\begin{equation}
\frac{1}{2}\left(  A+A^{T}\right)  =\left(
\begin{array}
[c]{cccccc}%
\alpha_{1} & \frac{\beta_{1}+\gamma_{1}}{2} & 0 & \cdots & \cdots & 0\\
\frac{\beta_{1}+\gamma_{1}}{2} & \alpha_{2} & \frac{\beta_{2}+\gamma_{2}}{2} &
\ddots &  & \vdots \\
0 & \frac{\beta_{2}+\gamma_{2}}{2} & \alpha_{1} & \frac{\beta_{1}+\gamma_{1}%
}{2} & \ddots & \vdots \\
\vdots & \ddots & \frac{\beta_{1}+\gamma_{1}}{2} & \alpha_{2} & \frac
{\beta_{2}+\gamma_{2}}{2} & 0\\
\vdots &  & \ddots & \frac{\beta_{2}+\gamma_{2}}{2} & \ddots & \ddots \\
0 & \cdots & \cdots & 0 & \ddots & \ddots
\end{array}
\right)  . \label{Para3}%
\end{equation}

Now, substituting (\ref{Para3}) in (\ref{Para2}) yields the set of $m-1$
conditions%
\[
\left \{
\begin{array}
[c]{ll}%
\text{for }i=1: & \sqrt{\alpha_{1}\alpha_{2}}>\left(  \beta_{1}+\gamma
_{1}\right)  \cos \left(  \frac{\pi}{m+1}\right) \\
\text{for }i=2: & \sqrt{\alpha_{1}\alpha_{2}}>\left(  \beta_{2}+\gamma
_{2}\right)  \cos \left(  \frac{\pi}{m+1}\right) \\
\text{for }i=3: & \sqrt{\alpha_{1}\alpha_{2}}>\left(  \beta_{1}+\gamma
_{1}\right)  \cos \left(  \frac{\pi}{m+1}\right) \\
\multicolumn{1}{c}{\vdots} & \multicolumn{1}{c}{\vdots}\\
\text{for }i=m-1: & \left \{
\begin{array}
[c]{ll}%
\sqrt{\alpha_{1}\alpha_{2}}>\left(  \beta_{2}+\gamma_{2}\right)  \cos \left(
\frac{\pi}{m+1}\right)  , & \text{if }m\text{ is odd}\\
\sqrt{\alpha_{1}\alpha_{2}}>\left(  \beta_{1}+\gamma_{1}\right)  \cos \left(
\frac{\pi}{m+1}\right)  , & \text{if }m\text{ is even.}%
\end{array}
\right.
\end{array}
\right.
\]

However, we notice that the $m-1$ conditons reduce to only $2$, which can be
combined to form condition (\ref{Para1}).
\end{proof}

\section{Existence of Solutions\label{SecMain}}

This section shows how the proposed system can be diagonalized using the
eigenvectors derived in Section \ref{SecEigen} above. We start by examining
the invariant regions of the system and then move to diagonalize the system
and establish the local and global existence of solutions given the initial
data lies within the invariant regions.

\subsection{Invariant Regions}

Let us denote the positive and descendingly ordered eigenvalues of matrix
$A^{T}$ by $\lambda_{\ell}$, with $\ell=1,...,m$, and the corresponding
eigenvectors by $V_{\ell}=\left(  v_{1\ell},...,v_{m\ell}\right)  ^{T}$, where
$\lambda_{1}>\lambda_{2}>...>\lambda_{m}$. Assuming the proposed system
satisfies the parabolicity condition (\ref{Para1}), matrix $A^{T}$ is
guaranteed to have strictly positive eigenvalues, and thus is unitarily
diagonalizable. Generally, the diagonalizing matrix can be formed containing
as its columns the normalized eigenvectors of $A$. Recalling that for every
eigenvalue there exist two eigenvectors with unit norm and opposite
directions, we can define the diagonalizing matrix as%
\begin{equation}
P=\left(  \left(  -1\right)  ^{i_{1}}V_{1}\shortmid \left(  -1\right)  ^{i_{2}%
}V_{2}\shortmid...\shortmid \left(  -1\right)  ^{i_{m}}V_{m}\right)  ,
\label{3.1}%
\end{equation}
where each power $i_{\ell}$\ is either equal to $1$ or $2$. In order to
simplify the notation, let us consider the two disjoint sets
\[
\mathfrak{Z=}\left \{  \ell|i_{\ell}=1\right \}
\]
and%
\[
\mathfrak{L=}\left \{  \ell|i_{\ell}=2\right \}  ,
\]
which satisfy the properties%
\begin{equation}
\mathfrak{L}\cap \mathfrak{Z}=\phi \text{ and }\mathfrak{L}\cup \mathfrak{Z}%
=\left \{  1,2,...,m\right \}  . \label{3.3}%
\end{equation}

Each permutation of $\mathfrak{Z}$ and $\mathfrak{L}$ satisfying (\ref{3.3})
yields a valid diagonalizing matrix. The total number of possible permutations
is thus $2^{m}$, which is also the number of invariant regions $\Sigma
_{\mathfrak{L},\mathfrak{Z}}$ for the proposed system. These regions may be
written as%
\begin{equation}
\Sigma_{\mathfrak{L},\mathfrak{Z}}:=\left \{  U_{0}\in \mathbb{%
\mathbb{R}
}^{m}:\left \langle V_{z},U_{0}\right \rangle \leq0\leq \left \langle V_{\ell
},U_{0}\right \rangle ,\text{ }\ell \in \mathfrak{L},\text{ }z\in \mathfrak{Z}%
\right \}  , \label{3.4}%
\end{equation}
subject to%
\begin{equation}
\left \langle V_{z},B\right \rangle \leq0\leq \left \langle V_{\ell}%
,B\right \rangle ,\text{ }\ell \in \mathfrak{L},\ z\in \mathfrak{Z}. \label{3.5}%
\end{equation}

For simplicity, we will only consider one of the invariant regions which
corresponds to the sets $\mathfrak{L}=\left \{  1,2,...,m\right \}  $ and
$\mathfrak{Z}=\emptyset$ and is defined in (\ref{InvReg1}) and (\ref{InvReg2}%
). This yields the diagonalizing matrix%
\begin{equation}
P=\left(  V_{1}\shortmid V_{2}\shortmid...\shortmid V_{m}\right)  .
\label{3.6}%
\end{equation}

Note that the work carried out in the following subsections can be trivially
extended to the remaining $2^{m}-1$ regions.

\subsection{Diagonalization and Local Existence of
Solutions\label{SubSecLocal}}

In order to establish the local existence of solutions for the proposed system
(\ref{Sys1}), we start by diagonalizing the system by means of the
diagonalizing matrix defined in (\ref{3.6}). We follow the same work performed
in \cite{Abdelmalek2016} to obtain the equivalent diagonal system. First, let%
\begin{equation}
W=\left(  w_{1},w_{2},...,w_{m}\right)  ^{T}=P^{T}U, \label{4.1}%
\end{equation}
where%
\begin{align*}
w_{\ell}  &  :=\left \langle V_{\ell},U\right \rangle \\
&  =\left \{
\begin{array}
[c]{ll}%
\left \langle V_{\ell},U\right \rangle , & \ell \in \mathfrak{L}\\
\left \langle \left(  -1\right)  V_{\ell},U\right \rangle , & \ell
\in \mathfrak{Z.}%
\end{array}
\right.
\end{align*}

Let us also define the functional%
\begin{equation}
\digamma \left(  W\right)  =\left(  \digamma_{1},\digamma_{2},...,\digamma
_{m}\right)  ^{T}=P^{T}F\left(  U\right)  , \label{4.2}%
\end{equation}
with each function%
\[
\digamma_{\ell}:=\left \langle V_{\ell},F\right \rangle
\]
fulfilling the following conditions:

\begin{enumerate}
\item[(A1)] Must be continuously differentiable on $%
\mathbb{R}
_{+}^{m}$ for all $\ell=1,...,m$, satisfying $\digamma_{\ell}(w_{1}%
,...,w_{\ell-1},0,w_{\ell+1},...,w_{m})\geq0$, for all $w_{\ell}\geq0;$
$\ell=1,...,m$.

\item[(A2)] Must be of polynomial growth (see the work of Hollis and Morgan
\cite{Hollis1994}), which means that for all $\ell=1,...,m$:%
\begin{equation}
\left \vert \digamma_{\ell}\left(  W\right)  \right \vert \leq C_{1}\left(
1+\left \langle W,1\right \rangle \right)  ^{N},N\in%
\mathbb{N}
\text{,on }\left(  0,+\infty \right)  ^{m}. \label{4.2.1}%
\end{equation}

\item[(A3)] Must satisfy the inequality:%
\begin{equation}
\left \langle D,\digamma \left(  W\right)  \right \rangle \leq C_{2}\left(
1+\left \langle W,1\right \rangle \right)  , \label{4.2.2}%
\end{equation}
where%
\[
D:=\left(  D_{1},D_{2},...,D_{m-1},1\right)  ^{T},
\]
for all $w_{\ell}\geq0,$ $\ell=1,...,m,$. All the constants $D_{\ell}$ satisfy
$D_{\ell}\geq \overline{D_{\ell}},$ $\ell=1,...,m$ where $\overline{D_{\ell}},$
$\ell=1,...,m,$ are sufficiently large positive constants.
\end{enumerate}

Note that $C_{1}$ and $C_{2}$ are uniformly bounded positive functions defined
on $\mathbb{R}_{+}^{m}$.

Finally, let%
\[
\Lambda=P^{T}B.
\]
Now, by observing the similarity transformation%
\begin{align}
P^{T}A\left(  P^{T}\right)  ^{-1}  &  =\left(  P^{-1}A^{T}P\right)
^{T}\nonumber \\
&  =\operatorname*{diag}(\lambda_{1},\lambda_{2},...,\lambda_{m}), \label{4.3}%
\end{align}
we can propose the following:

\begin{proposition}
\label{DiagProp}Diagonalizing system (\ref{Sys1}) by means of $P^{T}$\ yields%
\begin{equation}
W_{t}-\operatorname*{diag}(\lambda_{1},\lambda_{2},...,\lambda_{m})\Delta
W=\digamma \left(  W\right)  \text{ \ in }\Omega \times \left(  0,+\infty \right)
\label{EqSys}%
\end{equation}
with the boundary condition%
\begin{equation}
\alpha W+\left(  1-\alpha \right)  \partial_{n}W=\Lambda \text{ \ on }%
\partial \Omega \times \left(  0,+\infty \right)  \label{EqBound1}%
\end{equation}
or%
\begin{equation}
\alpha W+\left(  1-\alpha \right)  \operatorname*{diag}(\lambda_{1},\lambda
_{2},...,\lambda_{m})\partial_{n}W=\Lambda, \label{EqBound2}%
\end{equation}
and the initial data%
\begin{equation}
W\left(  x,0\right)  =W_{0}\text{ \ on }\Omega. \label{EqInit1}%
\end{equation}

\end{proposition}

The proof of Proposition \ref{DiagProp} is trivial and can be looked up in
\cite{Abdelmalek2016}. The diagonal system in (\ref{EqSys}) is equivalent to
(\ref{Sys1}) in the invariant region given in (\ref{InvReg1}) and
(\ref{InvReg2}).

By considering the equivalent diagonal system in (\ref{EqSys}), we can now
establish the local existence and uniqueness of solutions for the original
system (\ref{Sys1}) with initial data in $C(\overline{\Omega})$ or
$L^{p}(\Omega)$, $p\in \left(  1,+\infty \right)  $ using the basic existence
theory for abstract semilinear differential equations (Friedman
\cite{Friedman1964}, Henry \cite{Henry1984} and Pazy \cite{Pazy1983}). It
simply follows that the solutions are classical on $\left(  0,T_{\max}\right)
$, with $T_{\max}$ denoting the eventual blow up time in $L^{\infty}(\Omega)$.
The local solution is continued globally by \textit{apriori} estimates.

\subsection{Global Existence of Solutions}

The aim here is to establish the global existence of solutions for the
equivalent system (\ref{EqSys}) and consequently the original system
(\ref{Sys1}) subject to the parabolicity condition (\ref{Para1}) through the
use of an appropriate Lyapunov functional. The results obtained here are
similar to those of \cite{Abdelmalek2016}. Hence, no detailed proofs will be
given here.

Let us define%
\begin{equation}
K_{l}^{r}=K_{r-1}^{r-1}K_{l}^{r-1}-\left[  H_{l}^{r-1}\right]  ^{2},\text{
}r=3,...,l\text{,} \label{mycond}%
\end{equation}
where%
\[
H_{l}^{r}=\underset{1\leq \ell,\kappa \leq l}{\det}\left(  \left(
a_{\ell,\kappa}\right)  _{\substack{\ell \neq l,...r+1\\ \kappa \neq
l-1,..r}}\right)  \overset{k=r-2}{\underset{k=1}{\Pi}}\left(  \det \left[
k\right]  \right)  ^{2^{\left(  r-k-2\right)  }},\text{ }r=3,...,l-1\text{,}%
\]%
\[
K_{l}^{2}=\underset{\text{positive value}}{\underbrace{\lambda_{1}\lambda
_{l}\overset{l-1}{\underset{k=1}{\Pi}}\theta_{k}^{2\left(  p_{k}+1\right)
^{2}}\overset{m-1}{\underset{k=l}{\Pi}}\theta_{k}^{2\left(  p_{k}+2\right)
^{2}}}}\left[  \overset{l-1}{\underset{k=1}{\Pi}}\theta_{k}^{2}-A_{1l}%
^{2}\right]  ,
\]
and%
\[
H_{l}^{2}=\underset{\text{positive value}}{\underbrace{\lambda_{1}%
\sqrt{\lambda_{2}\lambda_{l}}\theta_{1}^{2\left(  p_{1}+1\right)  ^{2}%
}\overset{l-1}{\underset{k=2}{\Pi}}\theta_{k}^{\left(  p_{k}+2\right)
^{2}+\left(  p_{k}+1\right)  ^{2}}\overset{m-1}{\underset{k=l}{\Pi}}\theta
_{k}^{2\left(  p_{k}+2\right)  ^{2}}}}\left[  \theta_{1}^{2}A_{2l}%
-A_{12}A_{1l}\right]  .
\]

The term $\underset{1\leq \ell,\kappa \leq l}{\det}\left(  \left(
a_{\ell,\kappa}\right)  _{\substack{\ell \neq l,...r+1\\ \kappa \neq
l-1,..r}}\right)  $ denotes the determinant of the $r$ square symmetric matrix
obtained from $\left(  a_{\ell,\kappa}\right)  _{1\leq \ell,\kappa \leq m}$ by
removing the $\left(  r+1\right)  ^{th},\left(  r+2\right)  ^{\text{th}%
},...,l^{\text{th}}$ rows and the $r^{\text{th}},\left(  r+1\right)
^{\text{th}},...,\left(  l-1\right)  ^{\text{th}}$ columns. where $\det \left[
1\right]  ,...,\det \left[  m\right]  $ \ are the minors of the matrix $\left(
a_{\ell,\kappa}\right)  _{1\leq \ell,\kappa \leq m}.$ The elements of the matrix
are:%
\begin{equation}
a_{\ell \kappa}=\frac{\lambda_{\ell}+\lambda_{\kappa}}{2}\theta_{1}^{p_{1}^{2}%
}...\theta_{\left(  \ell-1\right)  }^{p_{\left(  \ell-1\right)  }^{2}}%
\theta_{\ell}^{\left(  p_{\ell}+1\right)  ^{2}}...\theta_{\kappa-1}^{\left(
p_{\left(  \kappa-1\right)  }+1\right)  ^{2}}\theta_{\kappa}^{\left(
p_{\kappa}+2\right)  ^{2}}...\theta_{\left(  m-1\right)  }^{\left(  p_{\left(
m-1\right)  }+2\right)  ^{2}}. \label{1.13}%
\end{equation}
where $\lambda_{\ell}$ in (\ref{2.2})-(\ref{2.3}). Note that $A_{\ell \kappa
}=\dfrac{\lambda_{\ell}+\lambda_{\kappa}}{2\sqrt{\lambda_{\ell}\lambda
_{\kappa}}}$ for all $\ell,\kappa=1,...,m$, and $\theta_{\ell};$
$\ell=1,...,\left(  m-1\right)  $ are positive constants.

\begin{theorem}
\label{GlobalTheo}Suppose that the functions $\digamma_{\ell};$ $\ell=1,...,m$
are of polynomial growth and satisfy condition (\ref{4.2.2}) for some positive
constants $D_{\ell};$ $\ell=1,...,m$ sufficiently large. Let $\left(
w_{1}\left(  t,.\right)  ,w_{2}\left(  t,.\right)  ,...,w_{m}\left(
t,.\right)  \right)  $ be a solution of (\ref{EqSys}) and%
\begin{equation}
L(t)=\int_{\Omega}H_{p_{m}}\left(  w_{1}\left(  t,x\right)  ,w_{2}\left(
t,x\right)  ,...,w_{m}\left(  t,x\right)  \right)  dx, \label{5.3}%
\end{equation}
where%
\[
H_{p_{m}}\left(  w_{1},...,w_{m}\right)  =\overset{p_{m}}{\underset{p_{m-1}%
=0}{\sum}}...\overset{p_{2}}{\underset{p_{1}=0}{\sum}}C_{p_{m}}^{p_{m-1}%
}...C_{p_{2}}^{p_{1}}\theta_{1}^{p_{1}^{2}}...\theta_{\left(  m-1\right)
}^{p_{\left(  m-1\right)  }^{2}}w_{1}^{p_{1}}w_{2}^{p_{2}-p_{1}}%
...w_{m}^{p_{m}-p_{m-1}},
\]
with $p_{m}$ a positive integer and $C_{p_{\kappa}}^{p_{\ell}}=\frac
{p_{\kappa}!}{p_{\ell}!\left(  p_{\kappa}-p_{\ell}\right)  !}$.\newline Also
suppose that the following condition is satisfied%
\begin{equation}
K_{l}^{l}>0;\text{ }l=2,...,m\text{,} \label{1.12}%
\end{equation}
It follows from these conditions that the functional $L$ is uniformly bounded
on the interval $\left[  0,T^{\ast}\right]  ,$ $T^{\ast}<T_{\max}$.
\end{theorem}

\begin{corollary}
\label{corollary1}Under the assumptions of theorem \ref{GlobalTheo}, all
solutions of (\ref{EqSys}) with positive initial data in $L^{\infty}\left(
\Omega \right)  $ are in $L^{\infty}\left(  0,T^{\ast};L^{p}\left(
\Omega \right)  \right)  $ for some $p\geq1$.
\end{corollary}

\begin{proposition}
\label{proposition2}Under the assumptions of theorem \ref{GlobalTheo} and
given that the condition (\ref{Para1}) is satisfied, all solutions of
(\ref{EqSys}) with positive initial data in $L^{\infty}\left(  \Omega \right)
$ are global for some $p>\dfrac{MN}{2}$.
\end{proposition}

\section{Numerical Example\label{SecEx}}

In order to put the findings of this study to the test, let us consider the
following $5$-component system%
\begin{equation}
\dfrac{\partial U}{\partial t}-A\Delta U=F\left(  U\right)  , \label{Ex2}%
\end{equation}
where the transposed diffusion matrix is given by%
\begin{equation}
A^{T}=\left(
\begin{array}
[c]{ccccc}%
1 & 0.5 & 0 & 0 & 0\\
0.3 & 1.5 & 0.7 & 0 & 0\\
0 & 0.25 & 1 & 0.5 & 0\\
0 & 0 & 0.3 & 1.5 & 0.7\\
0 & 0 & 0 & 0.25 & 1
\end{array}
\right)  , \label{6.2}%
\end{equation}
and the reaction functional $F\left(  U\right)  $ is of the form%
\[
F\left(  U\right)  =\left(
\begin{array}
[c]{ccccc}%
F_{1} & F_{2} & F_{3} & F_{4} & F_{5}%
\end{array}
\right)  ^{T},
\]
with%
\[
F_{j}\left(  U\right)  =U^{T}\Upsilon_{j}U+\sigma_{j}^{T}U,\text{
\  \ }j=1,...,5.
\]
For the purpose of this example, let $\Upsilon_{j}$ be the symmetric matrices
given by%
\[
\Upsilon_{1}=\left(
\begin{array}
[c]{ccccc}%
0.0146 & -0.0257 & 0.0073 & -0.0088 & 0\\
-0.0257 & 0.0202 & -0.004 & 0 & 0.0044\\
0.0073 & -0.004 & 0.0005 & 0.0011 & -0.0015\\
-0.0088 & 0 & 0.0011 & -0.0043 & 0.0027\\
0 & 0.0044 & -0.0015 & 0.0027 & -0.0007
\end{array}
\right)  ,
\]%
\[
\Upsilon_{2}=\left(
\begin{array}
[c]{ccccc}%
0.1142 & 0.228 & 0.0571 & 0.1293 & 0\\
0.228 & -0.2281 & -0.0153 & 0 & -0.0646\\
0.0571 & -0.0153 & 0.0041 & 0.0158 & -0.0122\\
0.1293 & 0 & 0.0158 & 0.0489 & -0.0244\\
0 & -0.0646 & -0.0122 & -0.0244 & -0.0061
\end{array}
\right)  ,
\]%
\[
\Upsilon_{3}=\left(
\begin{array}
[c]{ccccc}%
0.3702 & -0.1245 & 0.1851 & 0.0194 & 0\\
-0.1245 & 0.0371 & -0.0817 & 0 & -0.0097\\
0.1851 & -0.0817 & 0.0132 & 0.0364 & -0.0397\\
0.0194 & 0 & 0.0364 & -0.0079 & 0.0133\\
0 & -0.0097 & -0.0397 & 0.0133 & -0.0198
\end{array}
\right)  ,
\]%
\[
\Upsilon_{4}=\left(
\begin{array}
[c]{ccccc}%
-0.1316 & -0.1013 & -0.0658 & -0.0743 & 0\\
-0.1013 & 0.1177 & 0.0236 & 0 & 0.0371\\
-0.0658 & 0.0236 & -0.0047 & -0.0154 & 0.0141\\
-0.0743 & 0 & -0.0154 & -0.0252 & 0.0108\\
0 & 0.0371 & 0.0141 & 0.0108 & 0.0070
\end{array}
\right)  ,
\]
and%
\[
\Upsilon_{5}=\left(
\begin{array}
[c]{ccccc}%
-0.1651 & 0.5295 & -0.0825 & 0.2108 & 0\\
0.5295 & -0.4429 & 0.0539 & 0 & -0.1054\\
-0.0825 & 0.0539 & -0.0059 & -0.0081 & 0.0177\\
0.2108 & 0 & -0.0081 & 0.0949 & -0.0567\\
0 & -0.1054 & 0.0177 & -0.0567 & 0.0088
\end{array}
\right)  .
\]
Also, suppose that%
\[
\left \{
\begin{array}
[c]{l}%
\sigma_{1}=\left(
\begin{array}
[c]{ccccc}%
0.0795 & 0.0303 & -0.0243 & -0.014 & 0.0059
\end{array}
\right)  ^{T}\\
\sigma_{2}=\left(
\begin{array}
[c]{ccccc}%
-0.6466 & -0.6144 & 0.0798 & 0.2844 & 0.0572
\end{array}
\right)  ^{T}\\
\sigma_{3}=\left(
\begin{array}
[c]{ccccc}%
0.4549 & -0.2791 & -0.2846 & 0.1292 & 0.1635
\end{array}
\right)  ^{T}\\
\sigma_{4}=\left(
\begin{array}
[c]{ccccc}%
0.2682 & 0.3879 & 0.0097 & -0.1796 & -0.0618
\end{array}
\right)  ^{T}\\
\sigma_{5}=\left(
\begin{array}
[c]{ccccc}%
-1.6033 & -0.8159 & 0.4251 & 0.3777 & -0.0608
\end{array}
\right)  ^{T}.
\end{array}
\right.
\]
The system clearly satisfies the parabolicity condition (\ref{Para1}) as%
\[
\frac{\sqrt{\alpha_{1}\alpha_{2}}}{\max \left \{  \beta_{1}+\gamma_{1},\beta
_{2}+\gamma_{2}\right \}  }=\frac{\sqrt{1.5}}{0.95}=1.2892>\cos \left(
\frac{\pi}{5}\right)  =0.8090.
\]
We have from (\ref{2.3})%
\begin{equation}
\beta=\sqrt{\frac{0.7\times0.25}{0.5\times0.3}}=1.0801. \label{6.4}%
\end{equation}
Hence, we can form the polynomial $p_{n}\left(  \mu \right)  $ as%
\[
\left \{
\begin{array}
[c]{l}%
p_{0}\left(  \mu \right)  =1,\text{ }p_{1}\left(  \mu \right)  =\mu \\
p_{2}\left(  \mu \right)  =\mu(\mu)-1=\mu^{2}-1,
\end{array}
\right.
\]
with solutions%
\begin{equation}
P_{1}=1\text{ and }P_{2}=-1. \label{6.5}%
\end{equation}

Now, the eigenvalues are $\alpha_{1}$ along with the solutions of the
following two equations derived from (\ref{2.4})%
\[
\left \{
\begin{array}
[c]{l}%
\frac{\left(  1-\lambda \right)  \left(  1.5-\lambda \right)  }{\sqrt
{0.5\times0.7\times0.3\times0.25}}-\frac{1}{1.0801}-1.0801=1\\
\frac{\left(  1-\lambda \right)  \left(  1.5-\lambda \right)  }{\sqrt
{0.5\times0.7\times0.3\times0.25}}-\frac{1}{1.0801}-1.0801=-1,
\end{array}
\right.
\]
which can be simplified to%
\[
\left \{
\begin{array}
[c]{l}%
6.1721\left(  \lambda-1\right)  \left(  \lambda-1.5\right)  -3.0059=0\\
6.172\,1\left(  \lambda-1\right)  \left(  \lambda-1.5\right)  -1.0059=0.
\end{array}
\right.
\]
Solving the two quadratic equations in $\lambda$ yields the four eigenvalues
of $A$, which in descending order can be given by%
\begin{equation}
\left \{
\begin{array}
[c]{l}%
\lambda_{1}=1.9913\\
\lambda_{2}=1.7248\\
\lambda_{3}=1\\
\lambda_{4}=0.77516\\
\lambda_{5}=0.50871.
\end{array}
\right.  \label{6.6}%
\end{equation}
Hence,%
\[
D=\operatorname*{diag}(\lambda_{1},\lambda_{2},\lambda_{3},\lambda_{4}%
,\lambda_{5}).
\]

Similarly, formula (\ref{2.7})-(\ref{2.7.1}) can be used to derive the
eigenvectors of $A^{T}$, which are arranged according to the corresponding
eigenvalues to form the diagonalizing matrix%
\begin{equation}
P=\left(
\begin{array}
[c]{ccccc}%
0.3848 & -0.5265 & -0.5632 & -0.9063 & -0.8769\\
0.7629 & -0.7633 & 0.5534 & 0.0000 & 0.3943\\
0.3705 & -0.0195 & -0.5423 & 0.3884 & -0.0325\\
0.3531 & 0.3534 & 0.2562 & 0.0000 & -0.1825\\
\underset{V_{1}}{\underbrace{0.0891}} & \underset{V_{2}}{\underbrace{0.1219}}
& \underset{V_{3}}{\underbrace{-0.1303}} & \underset{V_{4}}{\underbrace
{-0.1665}} & \underset{V_{5}}{\underbrace{0.2030}}%
\end{array}
\right)  . \label{6.7}%
\end{equation}
Matrix $P^{T}$ is used to diagonalize the system yields the equivalent system%
\begin{equation}
\left \{
\begin{array}
[c]{l}%
\frac{\partial w_{1}}{\partial t}-1.9913\Delta w_{1}=-0.5w_{1}w_{5}%
+0.65w_{2}\\
\frac{\partial w_{2}}{\partial t}-1.7248\Delta w_{2}=0.5w_{1}w_{5}-0.65w_{2}\\
\frac{\partial w_{3}}{\partial t}-\Delta w_{3}=-0.32w_{3}w_{5}+0.41w_{4}\\
\frac{\partial w_{4}}{\partial t}-0.77516\Delta w_{4}=0.32w_{3}w_{5}%
-0.41w_{4}\\
\frac{\partial w_{5}}{\partial t}-0.50871\Delta w_{5}=-0.5w_{1}w_{5}%
+0.65w_{2}-0.32w_{3}w_{5}+0.41w_{4}.
\end{array}
\right.  \label{Ex2Diag}%
\end{equation}
Note that for simplicity, we have neglected small terms and rounded the
polynomial coefficients to four decimal points. The resulting reaction terms
clearly satisfy conditions (A1) through (A3) as discussed in Section
\ref{SubSecLocal} above.

Observe that the proposed system has $2^{5}=32$ invariant regions where the
resulting $w_{\ell0}$ is guaranteed to be positive. We consider one of these
regions corresponding to $w_{\ell0}=\left \langle V_{\ell},U_{0}\right \rangle $
and given by%
\[
\Sigma_{\mathfrak{L},\emptyset}=\left \{  U_{0}\in \mathbb{%
\mathbb{R}
}^{m}:\left \langle V_{\ell},U_{0}\right \rangle \geq0,\text{ }\ell
=1,...,m\right \}  ,
\]
which yields five inequalities%
\begin{equation}
\left \{
\begin{array}
[c]{l}%
0.3848u_{01}+0.7629u_{02}+0.3705u_{03}+0.3531u_{04}+0.0891u_{05}\geq0\\
-0.5265u_{01}-0.7633u_{02}-0.0195u_{03}+0.3534u_{04}+0.1219u_{05}\geq0\\
-0.5632u_{01}+0.5534u_{02}-0.5423u_{03}+0.2562u_{04}-0.1303u_{05}\geq0\\
-0.9063u_{01}+0.3884u_{03}-0.1665u_{05}\geq0\\
-0.8769u_{01}+0.3943u_{02}-0.0325u_{03}-0.1825u_{04}+0.2030u_{05}\geq0,
\end{array}
\right.  \label{6.8}%
\end{equation}
with%
\[
U_{0}=\left(  u_{01},u_{02},u_{03},u_{04},u_{05}\right)  ^{T}.
\]
Solving this system of inequalities yields the first region where the initial
data is assumed to lie. We will consider for instance the initial data%
\begin{equation}
U_{0}=\left(  0,15,14,29,20\right)  ^{T}. \label{6.9}%
\end{equation}

The equivalent diagonalized system (\ref{Ex2Diag}) was solved numerically by
means of the finite difference (FD) method. Figures \ref{Fig1} and \ref{Fig2}
show the solutions to the diagonalized system (\ref{Ex2Diag}) and the original
system (\ref{Ex2}), respectively, in the diffusion free case. In the one
dimensional case, a sinusoidal perturbation is added to the initial data to
introduce spatial diversity into the model. The solutions are shown in Figures
\ref{Fig3} and \ref{Fig4}.

\begin{figure}[ptb]
\centering \includegraphics[width = 4in]{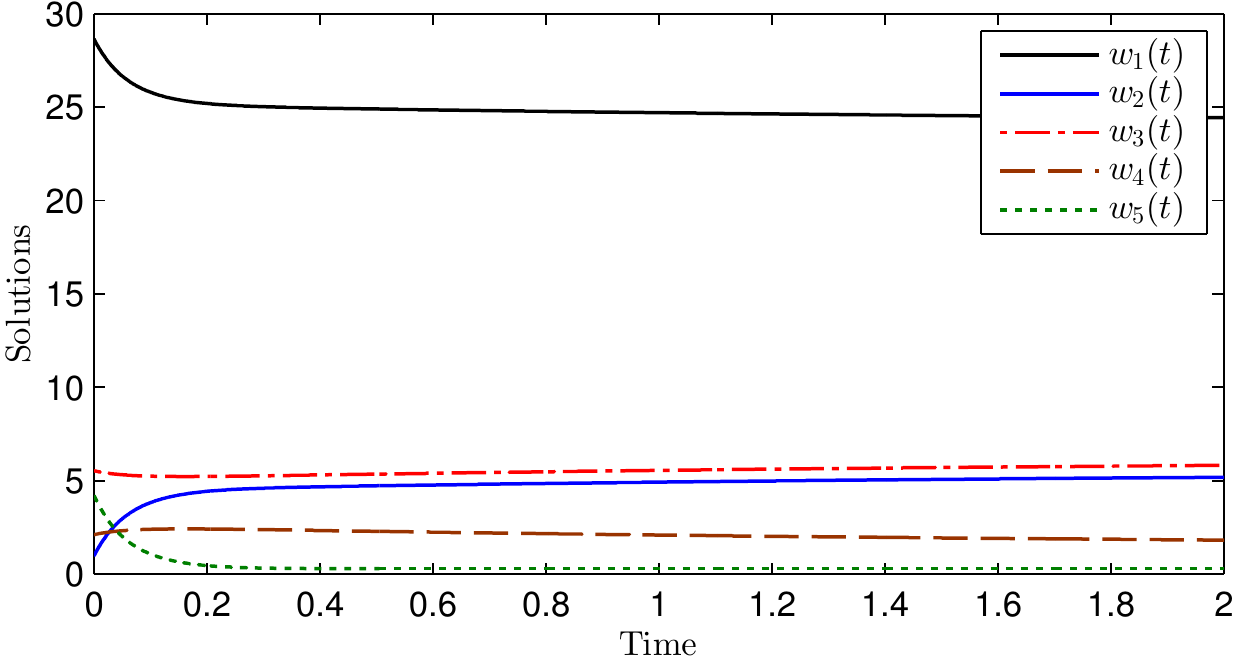}\caption{The
solutions of the equivalent diagonal system described in (\ref{Ex2Diag}) in
the diffusion-free case with the initial data given in (\ref{6.9}).}%
\label{Fig1}%
\end{figure}

\begin{figure}[ptb]
\centering \includegraphics[width = 4in]{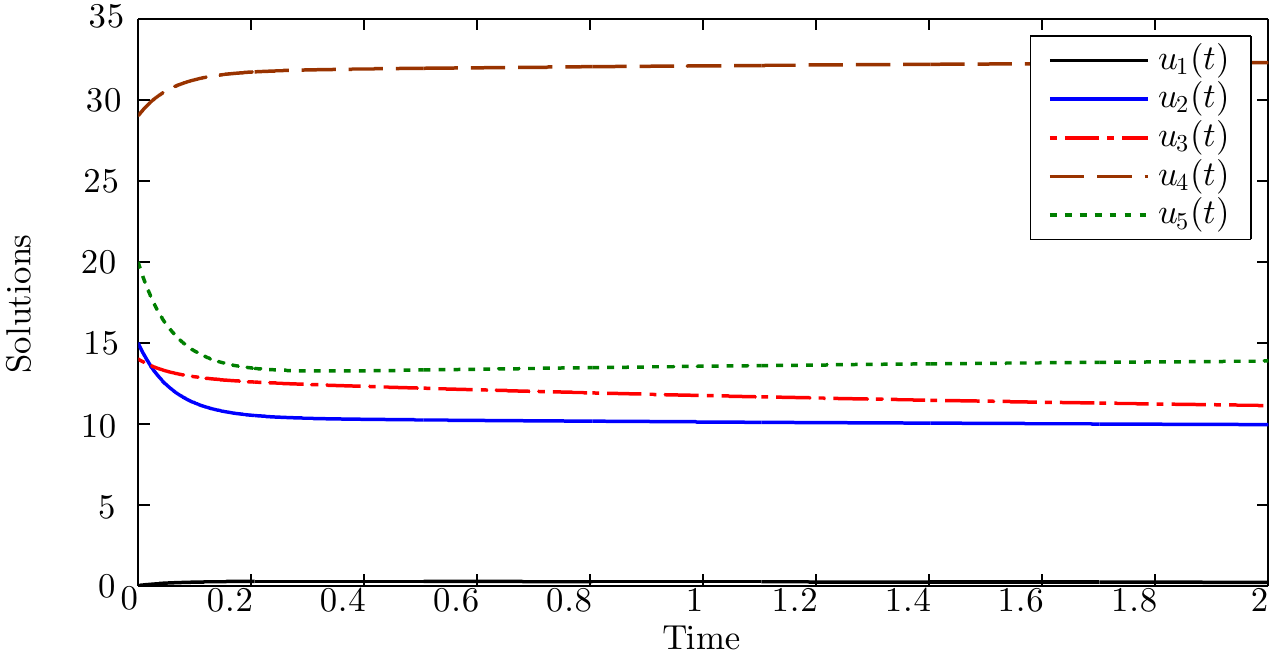}\caption{The
solutions of the original system described in (\ref{Ex2}) in the
diffusion-free case with the initial data given in (\ref{6.9}).}%
\label{Fig2}%
\end{figure}

\begin{figure}[ptb]
\centering \includegraphics[width = \textwidth]{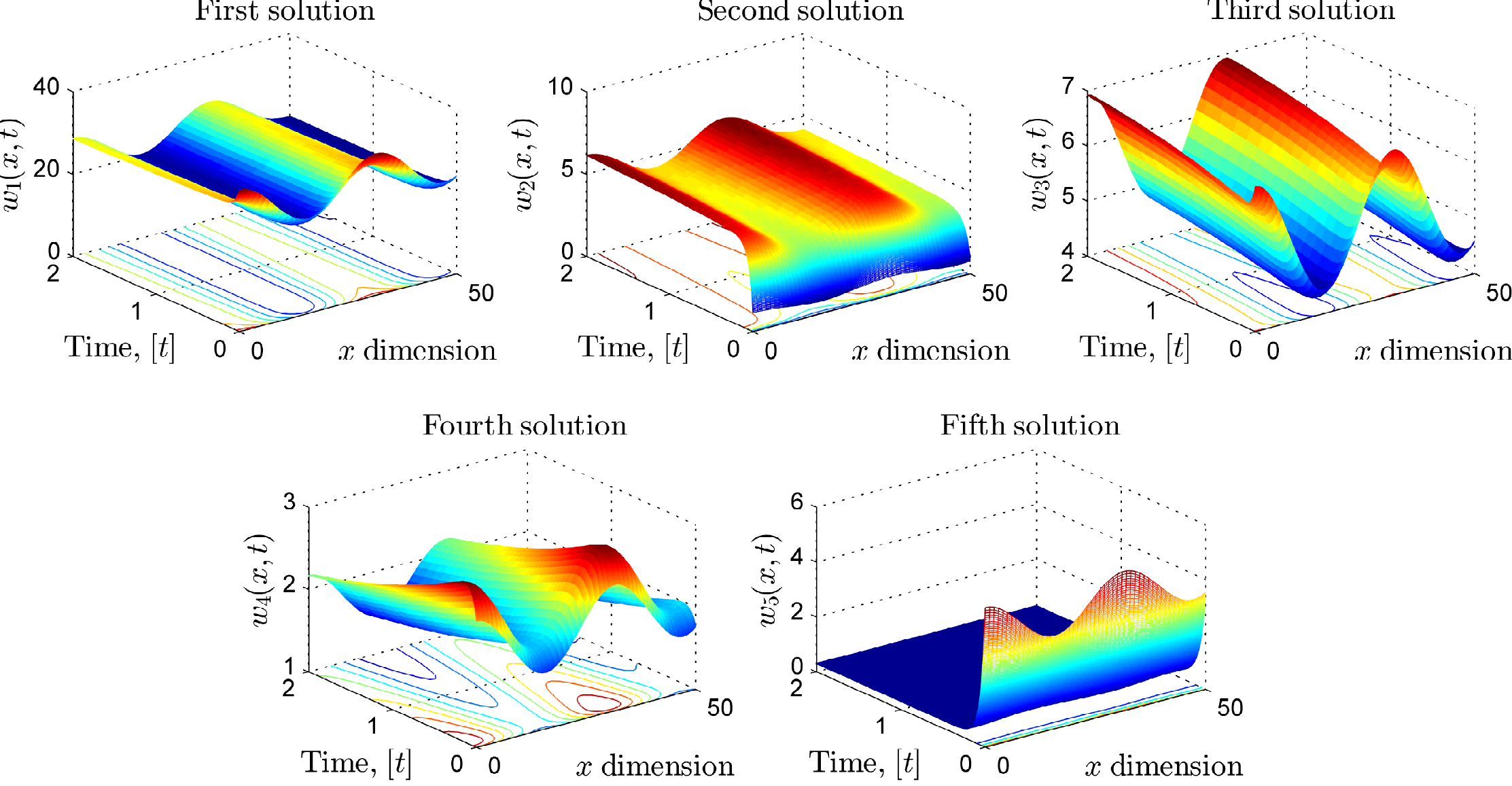}\caption{The
solutions of the equivalent diagonal system described in (\ref{Ex2Diag}) in
the one-dimensional diffusion case.}%
\label{Fig3}%
\end{figure}

\begin{figure}[ptb]
\centering \includegraphics[width = \textwidth]{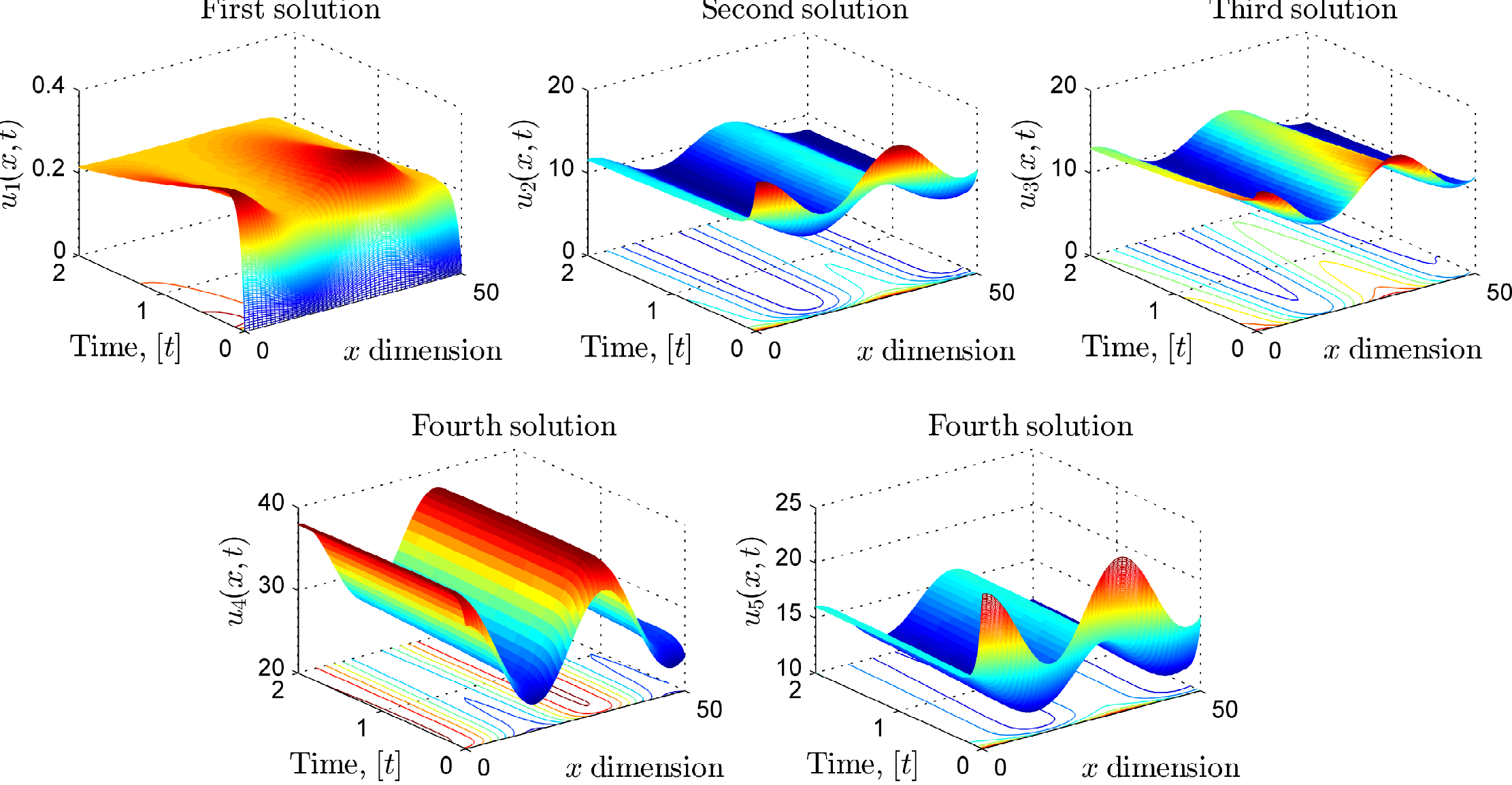}\caption{The
solutions of the original system described in (\ref{Ex2}) in the
one-dimensional diffusion case.}%
\label{Fig4}%
\end{figure}

\end{document}